\documentclass[10pt,reqno]{amsart}
\usepackage{amssymb,mathrsfs,graphicx}
\usepackage{ifthen}
\usepackage{colortbl}
\definecolor{black}{rgb}{0.0, 0.0, 0.0}
\definecolor{red}{rgb}{1.0, 0.5, 0.5}
\provideboolean{shownotes} 
\setboolean{shownotes}{true} 
\newcommand{\margnote}[1]{
\ifthenelse{\boolean{shownotes}}%
{\marginpar{\raggedright\tiny\texttt{#1}}}%
{}%
}
\newcommand{\hole}[1]{
\ifthenelse{\boolean{shownotes}}%
{\begin{center} \fbox{ \rule {.25cm}{0cm} \rule[-.1cm]{0cm}{.4cm}
\parbox{.85\textwidth}{\begin{center} \texttt{#1}\end{center}} \rule
{.25cm}{0cm}}\end{center}} {} }

\title[Emergent behavior of Cucker-Smale flocking particles with time delays]{Emergent behavior of Cucker-Smale flocking particles with time delays}

\author[Choi]{Young-Pil Choi}
\address[Young-Pil Choi]{\newline Department of Mathematics and Institute of Applied Mathematics\newline
Inha University, 402--751, Incheon, Republic of Korea}
\email{ypchoi@inha.ac.kr}

\author[Li]{Zhuchun Li}
\address[Zhuchun Li]{\newline Department of Mathematics \newline
Harbin Institute of Technology}
\email{lizhuchun@hit.edu.cn}

\numberwithin{equation}{section}

\newtheorem{theorem}{Theorem}[section]
\newtheorem{lemma}{Lemma}[section]

\newtheorem{remark}{Remark}[section]

\newcommand{\R}{\mathbb R}

\newcommand{\mc}{\mathcal C}

\newcommand{\bq}{\begin{equation}}
\newcommand{\eq}{\end{equation}}
\newcommand{\e}{\varepsilon}
\newcommand{\lt}{\left}
\newcommand{\rt}{\right}
\newcommand{\lal}{\langle}
\newcommand{\ral}{\rangle}

\newcommand{\mt}{\mathcal{T}}
\def\charf {\mbox{{\text 1}\kern-.30em {\text l}}}

\begin{document}
\allowdisplaybreaks

\date{\today}

\subjclass[]{}

\thanks{\textbf{Acknowledgments.} YPC was supported by National Research Foundation of Korea(NRF) grant funded by the Korea government(MSIP) (No. 2017R1C1B2012918 and 2017R1A4A1014735). The authors warmly thank Professor Seung-Yeal Ha for helpful discussion and valuable comments.
}

\begin{abstract}We analyze Cucker-Smale flocking particles with delayed coupling, where different constant delays are considered between particles. By constructing a system of dissipative differential inequalities together with a continuity argument, we provide a sufficient condition for the flocking behavior when the maximum value of time delays is sufficiently small.
\end{abstract}

\maketitle \centerline{\date}


%
%
%
%
\section{Introduction}
Let $(x_i(t), v_i(t)) \in \R^d \times \R^d, i=1,\dots, N$ be position and velocity at time $t$ of the $i$-th agent. Then the delayed Cucker-Smale particle system can be described by
\begin{align}\label{main_eq}
\begin{aligned}
\frac{d x_i(t)}{dt} &= v_i(t), \quad i =1,\cdots, N, \quad t >0,\cr
\frac{d v_i(t)}{dt} &= \frac1N \sum_{j\neq i } \psi(|x_j(t - \tau_{ji}) - x_i(t)|)(v_j(t - \tau_{ji}) - v_i(t)),
\end{aligned}
\end{align}
subject to the initial data:
\bq\label{eq_ini}
(x_i(s), v_i(s))  =(x^0_{i}(s), v^0_{i}(s)), \quad i =1,\dots, N, \quad s\in [-\tau, 0].
\eq
Here, $\psi: \R_+ \to \R_+$ is a communication weight function,    $\tau_{ji} > 0$ denotes the interaction delay between $i$-th and $j$-th agents.

The main purpose of this paper is to study the effect of time delays in Cucker-Smale flocking particle system. For the proof, inspired by \cite{CH_p, HL}, we construct a system of dissipative differential inequalities by using diameters of position and velocity. Using that together with a continuity argument, we provide a sufficient condition for the flocking behavior estimate under a smallness assumption on the time delays. 

It is worth mentioning that there are a few literature on the flocking of Cucker-Smale type models with time delays. For example, a sufficient flocking condition for the Motsch-Tadmor variant of the model with processing delay is obtained in \cite{LW}, see also \cite{MT} for that model without time delays. In \cite{EHS},  sufficient flocking condition for the Cucker-Smale model with noise and delay is derived in terms
of noise intensity and delay length. In \cite{CH17} the first author and his collaborator analyzed a Cucker-Smale model with delay and
normalized communication weights where the communication weights received by any agent sum to 1. In another recent paper \cite{PT}, the authors considered the Cucker-Smale model with processing time-varying delays only in velocities, in  which the velocity is  governed by \[\frac{d v_i(t)}{dt}  = \frac1N \sum_{j\neq i } \psi(|x_j(t - \tau(t)) - x_i(t-\tau(t))|)(v_j(t - \tau(t)) - v_i(t)).\]
In comparison, in this paper we will consider the Cucker-Smale model \eqref{main_eq} with processing delays in velocities and positions. We also emphasize that the strategy used in \cite{PT} requires the strictly positive lower bound assumption for the weight function $\psi$, which is not needed in our framework. We refer to \cite{CHL17, CCP17} and references therein for recent surveys on Cucker-Smale type flocking models.

We now introduce the main assumptions and state the main result.

{\bf Assumption on $\psi$.-} The communication weight $\psi$ is bounded, positive, non increasing and Lipschitz continuous on $\R_+$, with $\psi(0) = 1$.

{\bf Assumption on $\tau_{ji}$.-} The interaction delays are strictly positive and symmetric, i.e., $\tau_{ji} = \tau_{ij} > 0$ for all $i,j \in \{1,\cdots, N\}$  and  $\tau:=\max_{1 \leq i,j\leq N}\tau_{ji} < \infty$.

\begin{theorem}\label{thm_main}Let $\{(x_i, v_i)\}_{i=1}^N$ be a global solution to the system \eqref{main_eq}-\eqref{eq_ini}. Suppose that  there exist     some constants $\tau_0 > 0$ and  $\alpha > 0$    such that
\begin{equation}\label{cond}
\frac{d_V(0)}{  \alpha\psi(d_X(0) + R_v \tau_0 + \alpha)} <1 .
\end{equation}
Then there exists  {$\bar\tau \in (0, \tau_0) $}   such that for all $\tau \in (0,\bar\tau]$ we have
\[
\sup_{-\tau \leq  < +\infty}d_X(t) < +\infty \quad \mbox{and} \quad d_V(t) \leq C_0 e^{-c_1 \psi(d_X(0) + R_v\tau_0 + \alpha) t}, \quad \forall\,t \geq 0,
\]
where $C_0$ and $c_1$ are some positive constants.
\end{theorem}
%
\begin{remark}The communication weight $\psi$ introduced in the seminal paper by Cucker and Smale \cite{CS07} is of the form
\[
\psi(r) = \frac{1}{(1 + r^2)^{\beta/2}} \quad \mbox{with} \quad \beta > 0.
\]
Thus for the long-range communication weight, i.e., $\beta <1$, we can always find   positive constants  $\tau_0$ and $\alpha$  satisfying the assumption \eqref{cond}.
\end{remark}

In the next section of this note, we will present the details for the proof of main result.

\section{Emergent behavior: Proof of Theorem \ref{thm_main}}

\subsection{Global existence and uniqueness of solutions to the system \eqref{main_eq}} In this part, we first prove the global-in-time existence and uniqueness of solutions for the system \eqref{main_eq} so that all computations for the emergent behavior are justified. We notice from the above assumption on $\psi$ that the right-hand-side of \eqref{main_eq} is locally Lipschitz continuous as a function of $(x_i(t), v_i(t))$. Thus, by the Cauchy-Lipschitz theorem, the particle system \eqref{main_eq} admits a unique local-in-time $\mc^1$-solution. On the other hand, that local-in-time solution can be a global-in-time solution once we can show the uniform-in-time boundedness of the velocity since $\psi$ is bounded and Lipschitz.

In the lemma below, we show the uniform-in-time boundedness of the velocity which guarantees the global-in-time existence of the unique solution to the system \eqref{main_eq}. We are also going to use this  estimate for the large-time behavior.

\begin{lemma}\label{lem_bdd}Let $\{(x_i, v_i)\}_{i=1}^N$ be a solution to the system \eqref{main_eq}-\eqref{eq_ini}. Suppose that the initial velocity $v_{i0}, i=1,2,\cdots,N$ are continuous on the compact time interval $[-\tau,0]$  and denote $$R_v^\tau:=\max_{s \in [-\tau,0]}\max_{1 \leq i \leq N}|v_{i}(s)|  > 0.$$ Then we have
\[
R_v(t):= \max_{1 \leq i \leq N}|v_i(t)| \leq R_v^\tau \quad \mbox{for} \quad t \geq -\tau.
\]
\end{lemma}
\begin{proof} Although the proof is very similar to \cite[Lemma 2.2]{CH17}, we provide the details here for the completeness of this paper. For any $\e>0$, we set $R_v^{\tau,\e} := R_v^\tau + \e$ and $\mathcal{S}^\e := \{ t > 0 : R_v(t) < R^{\tau,\e}_v \mbox{ for } s \in [0,t)\}$. By the continuity of $R_v(t)$ together with $R_v^\tau < R_v^{\tau,\e}$, we get $\mathcal{S}^\e \neq 0$, and $T^\e_*:= \sup \mathcal{S}^\e > 0$ exists. We now claim $T^\e_* = \infty$. If not, it holds
\bq\label{con}
\lim_{t \to T^\e_* -} R_v(t) = R^\e_v, \quad \mbox{and} \quad R_v(t) < R^\e_v \quad \mbox{for} \quad t < T^\e_*.
\eq
On the other hand, it follows from \eqref{main_eq} that for $t < T^\e_*$,
$$\begin{aligned}
\frac12\frac{d |v_i(t)|^2}{dt} &= \frac1N\sum_{j\neq i} \psi(|x_j(t - \tau_{ji}) - x_i(t)|)(\lal v_j(t - \tau_{ji}), v_i(t)\ral - |v_i(t)|^2)\cr
&\leq  \frac1N\sum_{j\neq i} \psi(|x_j(t - \tau_{ji}) - x_i(t)|) {\lt(|v_j(t - \tau_{ji})| - |v_i(t)|\rt)|v_i(t)|} 
\cr
&\leq \frac1N\sum_{j\neq i} \psi(|x_j(t - \tau_{ji}) - x_i(t)|)\lt(R_v^\e - |v_i(t)|\rt)|v_i(t)|.
\end{aligned}$$
 This further yields
$$\begin{aligned}
\frac{d|v_i(t)|}{dt} \leq \frac1N\sum_{j\neq i} \psi(|x_j(t - \tau_{ji}) - x_i(t)|)\lt(R^\e_v - |v_i(t)|\rt) \leq R^\e_v - |v_i(t)|, \quad \mbox{a.e. on} \quad (0,T^\e_*),
\end{aligned}$$
due to $R^\e_v \geq R_v(t)$, i.e., $R^\e_v \geq |v_i(t)|$ for all $1\leq i \leq N$ and $t < T^\e_*$, and $0 \leq \psi \leq 1$. Applying Gronwall's inequality to the above, we have
\[
\lim_{t \to T^\e_*-} R_v(t) \leq \lt(R_v(0) - R^{\tau,\e}_v \rt)e^{-T^\e_*} + R^{\tau,\e}_v < R^{\tau,\e}_v,
\]
since $R_v(0) < R^{\tau,\e}_v$. This contradicts \eqref{con}, and thus $T^\e_* = \infty$. We finally pass to the limit $\e \to 0$ to conclude our desired result.
\end{proof}

\subsection{Construction of a Lyapunov functional} In this part, we construct the system of dissipative differential inequalities. For this, we introduce position and velocity diameters as 
\[
d_X(t) :=  \max_{1 \leq i,j \leq N}    |x_i(t) - x_j(t)| \quad \mbox{and} \quad d_V(t) := \max_{1 \leq i,j \leq N}|v_i(t) - v_j(t)|.
\]
%
%


\begin{lemma}\label{lem_sddi} Let $\{(x_i, v_i)\}_{i=1}^N$ be a global solution to the system \eqref{main_eq}-\eqref{eq_ini}. Then the diameters functions $d_X(t)$ and $d_V(t)$ satisfy
\begin{align}\label{eqn_dxdv}
\begin{aligned}
\frac{d}{dt} d_X(t) &\leq d_V(t), \cr
\frac{d}{dt} d_V(t) &\leq -  \psi(d_X(t) + R_v\tau) d_V(t) + 2\Delta_{N}^\tau(t),
\end{aligned}
\end{align}
for almost all $t > 0$, where $\Delta_{N}^\tau(t)$ is given by
\[
\Delta_{N}^\tau(t):=\frac1N\max_{1 \leq i \leq N}\sum_{k\neq i }|v_k(t - \tau_{ki}) - v_k(t)|
\]
and satisfies
\bq\label{eq_delta}
\Delta_N^\tau(t) \leq C_{N,1}\int_{t - \tau}^t d_V(s)\,ds + \int_{t-\tau}^t \Delta_N^\tau(s)\,ds,
\eq
for $t \geq \tau$, where $C_{N,1}:=(N-1)/N$.
\end{lemma}
\begin{proof}We first easily find from \eqref{main_eq} that
\[
\frac{d}{dt} d_X(t) \leq d_V(t).
\]
Next we derive the differential inequality for $d_V(t)$. Note that there exist at most countable number of increasing times $t_k$ such that we can choose indices $i$ and $j$ such that $d_V(t) = |v_i(t) - v_j(t)|$ on any time interval $(t_k, t_{k+1})$ since the number of particles is finite and continuity of the velocity trajectories. This allows us to estimate the time evolution of $d_V(t)$ as
$$\begin{aligned}
\frac12\frac{d}{dt}d_V(t)^2 &= \frac12\frac{d}{dt}|v_i(t) - v_j(t)|^2 \cr
&= \lt\lal v_i(t) - v_j(t), \frac{d v_i(t)}{dt} - \frac{dv_j(t)}{dt}\rt\ral\cr
&= \frac1N\lt\lal v_i(t) - v_j(t),  \sum_{k\neq i } \psi(|x_k(t - \tau_{ki}) - x_i(t)|)(v_k(t - \tau_{ki}) - v_i(t))\rt\ral\cr
&\quad - \frac1N\lt\lal v_i(t) - v_j(t),  \sum_{k\neq j } \psi(|x_k(t - \tau_{kj}) - x_j(t)|)(v_k(t - \tau_{kj}) - v_j(t))\rt\ral \cr
&=: I_1 + I_2.
\end{aligned}$$
Before estimating the terms $I_i,i=1,2$, we notice that
\[
\lt|x_k(t - \tau_{ki}) - x_i(t)\rt| = \lt| x_k(t) - x_i(t) + \int_t^{t - \tau_{ki}} v_k(s)\,ds \rt| \leq d_X(t) + R_v\tau.
\]
Using the above inequality, we estimate $I_1$ as
$$\begin{aligned}
I_1 &= \frac1N\sum_{k\neq i } \psi(|x_k(t - \tau_{ki}) - x_i(t)|)\lt\lal v_i(t) - v_j(t),v_k(t) - v_i(t)\rt\ral\cr
&\quad + \frac1N\sum_{k\neq i } \psi(|x_k(t - \tau_{ki}) - x_i(t)|)\lt\lal v_i(t) - v_j(t),v_k(t - \tau_{ki}) - v_k(t)\rt\ral \cr
&\leq \frac{\psi(d_X(t) + R_v\tau)}{N} \sum_{{k\neq i} } \lt\lal v_i(t) - v_j(t),v_k(t) - v_i(t)\rt\ral + \frac{d_V(t)}{N}\sum_{{k\neq i }}|v_k(t - \tau_{ki}) - v_k(t)|,
\end{aligned}$$
where we used $\psi \leq 1$ and   $\lt\lal v_i(t) - v_j(t),v_k(t) - v_i(t)\rt\ral \leq 0$  for   $(i,j)$ with $d_V = |v_i - v_j|$.  Similarly, we can obtain
$$\begin{aligned}
I_2 &\leq -\frac{\psi(d_X(t) + R_v\tau)}{N} \sum_{{k\neq  j} } \lt\lal v_i(t) - v_j(t),v_k(t) - v_j(t)\rt\ral + \frac{d_V(t)}{N}\sum_{{k\neq  j} }|v_k(t - \tau_{kj}) - v_k(t)|.
\end{aligned}$$
This yields
\[
\frac12\frac{d}{dt} d_V(t)^2 \leq  {-\psi(d_X(t) + R_v\tau)   d_V(t)^2} + \frac{2d_V(t)}{N}\max_{1 \leq i \leq N}\sum_{k\neq i }|v_k(t - \tau_{ki}) - v_k(t)|,
\]
for almost all $t \geq 0$. Thus we have
\[
\frac{d}{dt} d_V(t) \leq - \psi(d_X(t) + R_v\tau) d_V(t) + 2\Delta_{N}^\tau(t),
\]
for almost all $t \geq 0$. We next estimate the term $\Delta_{N}^\tau(t)$. Note that
\[
|v_k(t -\tau_{ki}) - v_k(t)| = \lt| \int_{t - \tau_{ki}}^t \frac{d v_k(s)}{ds}\,ds \rt| \leq \int_{t - \tau}^t \lt|\frac{d v_k(s)}{ds}\rt|\,ds.
\]
This gives
\bq\label{eqn_del}
\Delta_N^\tau(t) \leq \frac1N\sum_{k=1}^N\int_{t - \tau}^t \lt|\frac{d v_k(s)}{ds}\rt|\,ds   \quad \mbox{for} \quad t \geq 0.
\eq
On the other hand, it follows from \eqref{main_eq} that
\begin{align}\label{eqn_del2}
\begin{aligned}
\lt| \frac{dv_k(s)}{ds}\rt| &= \lt|\frac1N\sum_{\ell \neq k} \psi(|x_\ell(s - \tau_{\ell k}) - x_k(s)|)(v_\ell(s - \tau_{\ell k}) - v_k(s)) \rt|\cr
&\leq \frac{d_V(s)(N-1)}{N} + \frac1N\sum_{\ell \neq k}|v_\ell(s - \tau_{\ell k}) - v_\ell(s))|\cr
&\leq C_{N,1}d_V(s) + \Delta_N^\tau(s),  
\end{aligned}
\end{align}
for $t\geq 0$. Combining the above estimates \eqref{eqn_del} and \eqref{eqn_del2} concludes the desired result.
\end{proof}

\begin{remark}If there is no time delay, i.e., $\tau = 0$, then the differential inequality in Lemma \ref{lem_sddi} becomes the standard system of dissipative differential inequalities in \cite{HL}.
\end{remark}

\begin{remark}\label{rmk_del} It follows from Lemma \ref{lem_bdd} that
\[
\lt| \frac{dv_k(s)}{ds}\rt| = \lt|\frac1N\sum_{\ell \neq k} \psi(|x_\ell(s - \tau_{\ell k}) - x_k(s)|)(v_\ell(s - \tau_{\ell k}) - v_k(s)) \rt| \leq 2R_v^\tau.
\]
This gives the following estimate:
\[
\sup_{0 \leq t \leq \tau}\Delta_N^\tau(t) \leq \frac1N\sum_{k=1}^N\int_{t - \tau}^t \lt|\frac{d v_k(s)}{ds}\rt|\,ds \leq 2R_v^\tau \tau \to 0 \quad \mbox{as} \quad \tau \to 0.
\]
\end{remark}

%
%
%
%
\subsection{Proof of Theorem \ref{thm_main}}
We are going to use the differential inequalities \eqref{eqn_dxdv}  together with a continuity argument  to complete the proof of Theorem \ref{thm_main}.

Set
\[
\mathcal{T} := \lt\{ t \in [0,\infty) : d_X(s) < d_X(0) + \alpha \quad \mbox{for} \quad s \in [0,t)\rt\}.
\]
It is clear from the continuity of the function $d_X(t)$ that $\mt \neq \emptyset$, thus we can set $\mathcal{T}^\infty := \sup \mathcal{T}$.

$\bullet$ {\bf Step A.-} (Time-decay estimate of $d_V(t)$ and $\Delta_N^\tau(t)$): According to \eqref{cond}, we first choose some  positive constants $\beta>0$ and $0<c<1$ such that
\[
\frac{d_V(0)}{\psi(d_X(0) + R_v \tau_0 + \alpha)} + \frac{2\beta}{  1 - c  } < \alpha.
\]
Then we set
$$\begin{aligned}
\mt_* := \bigg\{ t \in [0,\mt^\infty) : d_V(s) &< \lt(d_V(0) + \frac{2\beta\psi^\infty}{ 1 - c  }\rt)e^{-c\psi^\infty s} \quad \mbox{and} \cr
 \Delta^\tau_N(s) &< \beta(\psi^\infty)^2 e^{-c \psi^\infty s} \quad \mbox{for} \quad s \in [0,t)\bigg\}.
\end{aligned}$$
where we denoted   $\psi^\infty := \psi(d_X(0) + R_v \tau_0 + \alpha)$ for notational simplicity. Note that $\mt_* \neq \emptyset$ for $\tau$ small enough. Indeed, we find
\bq\label{as_ini}
d_V(0) < d_V(0) + \frac{2\beta\psi^\infty}{1 - c } \quad \mbox{and} \quad \sup_{0 \leq t \leq \tau}\Delta^\tau_N(t) < \beta(\psi^\infty)^2e^{-c \psi^\infty \tau}
\eq
for $\tau > 0$ small enough such that $2R_v^\tau \tau e^{c\psi^\infty \tau}<  \beta (\psi^\infty)^2$ due to Remark \ref{rmk_del}. Thus by continuity of functions $d_V(t)$ and $\Delta^\tau_N(t)$ there exists $\tau_1 > 0$ such that $0 < \mt^\infty_* : =\sup\mt_*$ for all $\tau \in (0,\tau_1)$. Then in the rest of this step we are going to show $\mt^\infty_* = \mt^\infty$. Suppose that $0<\mt^\infty_* < \mt^\infty$. Then we have either
\bq\label{cont_eq}
\lim_{t \to \mt^\infty_*\mbox{-}}d_V(t) = \lt(d_V(0) + \frac{2\beta\psi^\infty}{1- c }\rt)e^{-c \psi^\infty \mt^\infty_*} \quad \mbox{or} \quad  \lim_{t \to \mt^\infty_*\mbox{-}}\Delta^\tau_N(t) = \beta(\psi^\infty)^2 e^{-c \psi^\infty \mt^\infty_*}.
\eq
On the one hand, it follows from Lemma \ref{lem_sddi} that
\[
\frac{d}{dt}d_V(t) \leq - \psi^\infty d_V(t) + 2\beta(\psi^\infty)^2 e^{-c \psi^\infty t}, \quad \mbox{a.e. } t \in [0,\mt^\infty_*).
\]
Applying Gronwall's inequality yields
\[
d_V(t) \leq d_V(0) e^{-  \psi^\infty t} + \frac{2\beta \psi^\infty}{1 - c }\lt(e^{-c\psi^\infty t} - e^{- \psi^\infty t}\rt), \quad \quad t \in [0,\mt^\infty_*).
\]
Taking the limit $t \to \mt^\infty_*$- to the above inequality gives
$$\begin{aligned}
\lim_{t \to \mt^\infty_*\mbox{-}} d_V(t) &\leq d_V(0) e^{-  \psi^\infty \mt^\infty_*} + \frac{2\beta \psi^\infty}{1 - c }\lt(e^{-c \psi^\infty \mt^\infty_*} - e^{- \psi^\infty \mt^\infty_*}\rt)\cr
&< \lt(d_V(0) + \frac{2\beta\psi^\infty}{1 - c }\rt)e^{-c \psi^\infty \mt^\infty_*}.
\end{aligned}$$
On the other hand, we find from \eqref{eq_delta} together with \eqref{as_ini} that
$$\begin{aligned}
\Delta^\tau_N(t) &\leq \lt(C_{N,1}\lt(d_V(0) + \frac{2\beta\psi^\infty}{1 - c }\rt) + \beta(\psi^\infty)^2\rt)\int_{t - \tau}^t e^{-c \psi^\infty s}\,ds\cr
&= \lt(C_{N,1}\lt(d_V(0) + \frac{2\beta\psi^\infty}{1- c }\rt) + \beta(\psi^\infty)^2\rt)\lt(\frac{e^{c \psi^\infty \tau} - 1}{c \psi^\infty}\rt)e^{-c \psi^\infty t}
\end{aligned}$$
for all $t \in [0,\mt^*_*)$. We then now choose $0 < \tau_2 < \tau_1$ such that
\[
\lt(C_{N,1}\lt(d_V(0) + \frac{2\beta\psi^\infty}{1 - c }\rt) + \beta(\psi^\infty)^2\rt)\lt(\frac{e^{c \psi^\infty \tau} - 1}{c \psi^\infty}\rt) < \beta(\psi^\infty)^2
\]
for $\tau \in (0,\tau_2)$. This together with taking the limit $t \to \mt^\infty_*$- yields
\[
\lim_{t \to \mt^\infty_*\mbox{-}}\Delta^\tau_N(t) < \beta(\psi^\infty)^2 e^{-c \psi^\infty \mt^\infty_*}.
\]
Hence both equalities \eqref{cont_eq} do not hold, and this concludes $\mt^\infty_* = \mt^\infty$. \newline

 $\bullet$ {\bf Step B.-} (Uniform-in-time bound estimate of $d_X(t)$): We are now ready to show that $\mt^\infty = \infty$ when $\tau > 0$ is small enough. Note that for $t \in [0,\mt^\infty)$ and $\tau \in (0,\tau_2)$ it holds
$$\begin{aligned}
d_X(t) &< d_X(0)+\alpha,\cr
d_V(t) &< \lt(d_V(0) + \frac{2\beta\psi^\infty}{1 - c }\rt)e^{-c \psi^\infty t},\cr
 \Delta^\tau_N(t) &< \beta(\psi^\infty)^2 e^{-c \psi^\infty t}.
\end{aligned}$$
Suppose not, i.e., $\mathcal{T}^\infty < \infty$, then we get
\[
\lim_{t \to \mt^\infty\mbox{-}}d_X(t) = d_X(0)+\alpha.
\]
On the other hand, it follows from Lemma \ref{lem_sddi} together with the above estimate that
$$\begin{aligned}
d_X(t) &\leq d_X(0) + \int_0^t d_V(s)\,ds\cr
&\leq d_X(0) +  \lt(d_V(0) + \frac{2\beta\psi^\infty}{1 - c }\rt)\int_0^t e^{-c \psi^\infty s}\,ds\cr
&= d_X(0) +  \lt(d_V(0) + \frac{2\beta\psi^\infty}{1 - c }\rt)\frac{1}{c \psi^\infty}\lt(1 - e^{c \psi^\infty t} \rt)
\end{aligned}$$
for $t \in [0,\mt^\infty)$. This gives
\[
\lim_{t \to \mt^\infty\mbox{-}}d_X(t) \leq d_X(0) +  \lt(d_V(0) + \frac{2\beta\psi^\infty}{1 - c }\rt)\frac{1}{c \psi^\infty}\lt(1 - e^{c \psi^\infty \mt^\infty} \rt) < d_X(0) + \alpha.
\]
This is a contradiction and yields $\mt^\infty = \infty$. \newline

 $\bullet$ {\bf Step C.-} (Exponential decay estimate of $d_V(t)$): By the discussion in {\bf Step A} and {\bf Step B}, we find $\mt^\infty_* = \mt^\infty = \infty$, that is, the following inequalities hold for $t\geq 0$:
\[
d_X(t) \leq d_X(0) + \alpha \quad \mbox{and} \quad d_V(t) \leq \lt(d_V(0) + \frac{2\beta\psi^\infty}{ 1 - c  }\rt)e^{-c\psi^\infty t},
\]
where $\psi^\infty, \beta$, and $c$ are appeared in {\bf Step A}. This completes the proof.

%
%
%
%

\end{document}